\documentclass[11pt]{article}
\usepackage{amsfonts, amsmath, amssymb, amsthm, enumitem, aliascnt}

\PassOptionsToPackage{obeyspaces}{url}
\usepackage[colorlinks=true,citecolor=blue,urlcolor=blue,linkcolor=blue,bookmarksopen=true]{hyperref}
\usepackage{breakurl}
\usepackage{tikz}
\usetikzlibrary{shapes.geometric,positioning}

\usepackage{etoolbox} 
\patchcmd{\thebibliography}{\leftmargin\labelwidth}{\leftmargin\labelwidth\addtolength\itemsep{-0.1\baselineskip}}{}{}

\author{Boris Bukh\thanks{Department of Mathematical Sciences, Carnegie Mellon University, Pittsburgh, PA 15213, USA. Supported in part by Sloan Research Fellowship and by U.S.\ taxpayers through NSF CAREER grant DMS-1555149.} \and Anish Sevekari\thanks{Supported in part by U.S.\ taxpayers through NSF CAREER grant DMS-1555149.}}

\title{Linear orderings of combinatorial cubes}
\date{}

\newtheorem{theorem}{Theorem}

\newaliascnt{lemma}{theorem}
\newtheorem{lemma}[lemma]{Lemma}
\aliascntresetthe{lemma}

\newaliascnt{corollary}{theorem}

\aliascntresetthe{corollary}

\newaliascnt{proposition}{theorem}
\newtheorem{proposition}[proposition]{Proposition}
\aliascntresetthe{proposition}

\newcommand*{\eqdef}{\stackrel{\text{\tiny{def}}}{=}}            
\newcommand*{\abs}[1]{\lvert #1\rvert}                           
\newcommand*{\veps}{\varepsilon}                                 
\newcommand*{\N}{\mathbb{N}}                                     

\usepackage{stmaryrd}
\def\lbrac{\llbracket}                                           
\def\rbrac{\rrbracket}
\DeclareMathOperator{\lex}{lex}                                  
\def\wf{\mathbf}                                              
\def\before{\mathrel{\triangleleft}}                             
\def\prece{\preccurlyeq}                                         

\newcommand{\prefw}{\reflectbox{\ensuremath{\skew{-3.5}\vec{\reflectbox{\ensuremath{w}}}}}}
\newcommand{\suffw}{\skew{-1.5}\vec{w}}

\def\D#1-{$\if\relax\detokenize{#1}\relax d\else (d#1)\fi$\nobreakdash-\hspace{0pt}}

\begin{document}

\maketitle

\begin{abstract}
We show that, for every linear ordering of $[2]^n$, there is a large subcube
on which the ordering is lexicographic. We use this to deduce that
every long sequence contains a long monotone subsequence supported
on an affine cube.

More generally, we prove an analogous result for linear orderings of $[k]^n$.
We show that, for every such ordering, there is a large subcube on which the
ordering agrees with one of approximately $\frac{(k-1)!}{2(\ln 2)^k}$~orderings.
\end{abstract}

\section{Statement of results}
\paragraph{Monotone subsequences.}
A classic result of Erd\H{o}s and Szekeres \cite{erdos_szekeres} asserts that
every sufficiently long sequence $(a_1,a_2,\dotsc)$ contains a subsequence
$(a_{i_1},\dotsc,a_{i_m})$ of length $m$ that is monotone. One may wonder
if one may strengthen this result by requiring that the set
of indices $\{i_1,\dotsc,i_m\}$ is an arithmetically structured set.

Our first result is such a strengthening. Before stating it,
we recall that \emph{an affine \D-cube} is a set of the form
$\{x_0+\veps_1 x_1+\dotsb+\veps_d x_d : \veps_1,\dotsc,\veps_d\in \{0,1\}\}$.
An affine \D-cube is \emph{proper} if it contains $2^d$ distinct elements.

\begin{theorem}\label{thm:es}
For every $d$, there exists $m$ such that every sequence of $m$ distinct real numbers
contains a monotone subsequence whose index set is a proper affine \D-cube. 
\end{theorem}

In this result one cannot replace affine cubes by arithmetic progressions.
More precisely, we observe the following.
\begin{proposition}\label{prop:counterexample}
There exist arbitrarily long sequences of distinct real numbers 
that contains no monotone subsequences whose index set is a $3$-term
arithmetic progression.
\end{proposition}
\begin{proof}
We construct such sequences, which we dub \emph{$3$-AP-free}, 
inductively. We start with any sequence of length $1$. For the induction step,
we note that if $(a_1,\dotsc,a_m)$
is any $3$-AP-free sequence, and $M>2\max_i \abs{a_i}$, then 
$(a_1,a_1+M,a_2,a_2+M,\dotsc,a_m,a_m+M)$ is also $3$-AP-free.
Indeed, suppose $\{i,j,k\}$ is a three-term arithmetic progression.
If the parities of $i,j,k$ are the same, the sequence $(a_i,a_j,a_k)$ is not monotone 
by induction. If the parities of $i,j,k$ are different, then $(a_i,a_j,a_k)$ is not monotone
by the choice of~$M$.
\end{proof}

\paragraph{Hales--Jewett-type result for $[2]^n$.}
A common way to prove Ramsey results on the integers is to deduce
them from abstract statements about high-dimensional cubes. For example,
in this way one deduces van der Waerden's theorem from the Hales--Jewett
theorem, and Szemer\'edi's theorem from the density Hales--Jewett theorem.
Our proof of \autoref{thm:es} also follows this pattern: we deduce
\autoref{thm:es} from a Ramsey result about linear orderings of $[2]^n$.

As our most general result will apply not only to $[2]^n$, but 
to $[k]^n$ for any $k\geq 2$, we introduce definitions at that
level of generality. 

We shall think of elements of $[k]^n$ as words of length~$n$ over alphabet~$[k]$. 
A \emph{\D-parameter word} is a word $p$ over alphabet $[k]\cup\{\wf{*_1},\dotsc,\wf{*_d}\}$
that contains each of $\wf{*_1},\dotsc,\wf{*_d}$ at least once. For any word $w$ of length $d$ 
(possibly over a different alphabet), we let $p\lbrac w\rbrac $ be the word obtained from $p$ by replacing each $\wf{*_i}$
by~$w_i$, for each $i$. For example, if $p=\wf{21*_1*_2*_13}$ and $w=\wf{31}$,
then $p\lbrac w\rbrac=\wf{213133}$. 

If $p$ is a \D-parameter word, then the set $\{p\lbrac w\rbrac : w\in [k]^d\}$
can be naturally regarded as a copy of $[k]^d$ inside $[k]^n$; we
thus call it a (combinatorial) \emph{\D-subcube}. Two \D-parameter words that
differ in a permutation of $\{\wf{*_1},\dotsc,\wf{*_d}\}$ induce the same \D-subcube.
Call a \D-parameter word $p$ \emph{canonical} if the first occurrence
of $\wf{*_1}$ precedes the first occurrence of $\wf{*_2}$, which in turn precedes
the first occurrence of $\wf{*_3}$, etc. 
Canonical words induce a \emph{canonical bijection} between $[k]^d$ and corresponding
\D-subcubes; we shall always use this bijection when identifying \D-subcubes
with~$[k]^d$.

If $p$ is a $D$\nobreakdash-parameter word of length $n$ and $p'$ is a \D-parameter word of length
$D$, then $p\lbrac p'\rbrac$ is a \D-parameter word of length $n$. Furthermore,
if both $p$ and $p'$ are canonical, then so is $p\lbrac p'\rbrac$. Hence,
if $C_1\supseteq C_2\supset \dotsb$  is a nested chain of subcubes of $[k]^n$,
we may use the canonical bijection to regard $C_{i+1}$ as a subcube of $C_i$.

The canonical bijection also allows us to regard a restriction of a linear ordering on $[k]^n$
to any \D-subcube as a linear ordering on~$[k]^d$.
Namely, let $\before$ be a linear ordering on $[k]^n$ and let $p$ be a canonical \D-parameter
word of length $n$. If $w,w'\in [k]^d$, we then set $w\before w'$ whenever $p\lbrac w\rbrac \before p\lbrac w'\rbrac$.

Given a linear ordering $<$ on $[k]$, the \emph{lexicographic ordering} on $[k]^n$ is defined
by setting $w<_{\lex} w'$ whenever $w_i<w_i'$, where $i$ is the least index such that $w_i\neq w_i'$.
Note that if $p$ is a canonical \D-parameter word, then $w<_{\lex}w'$ holds for $w,w'\in [k]^d$
if and only if $p\lbrac w\rbrac <_{\lex} p\lbrac w'\rbrac$. Hence, under the canonical bijection,
a restriction of a lexicographic ordering to a \D-subcube is a lexicographic ordering 
on~$[k]^d$.

\begin{theorem}\label{thm:two}
For every $d$ there exists $n$ with the following property: for every linear ordering $\before$
of $[2]^n$ there is a \D-subcube $C$ of $[2]^n$ such that the restriction of $\before$ to $C$ 
is the lexicographic ordering for one of the two linear orderings of~$[2]$.
\end{theorem}
\autoref{thm:es} is an easy consequence of \autoref{thm:two}. Indeed, let $m=3^n$
and define the projection map $\pi\colon [2]^n\to [m]$ by $\pi(w)\eqdef\sum_{i=1}^n w_i3^{n-i}$.
The sequence $a_1,\dotsc,a_m$ then induces a linear ordering on $[2]^n$, where
$w\before w'$ whenever $a_{\pi(w)}<a_{\pi(w')}$. A lexicographically ordered \D-subcube
of $[2]^n$ then corresponds to a monotone subsequence of $a_1,\dotsc,a_n$ whose index set is 
a proper affine \D-cube.

\paragraph{Hales--Jewett-type result for general $[k]^n$.}
The naive generalization of \autoref{thm:two} to the case of $[k]^n$, with $k\geq 3$,
is false. As an example, define a linear ordering $\before$ on $[3]^n$ as follows:
for a word $w$ let $w_{12}$ be the word obtained from $w$ by replacing each $\wf{1}$
by $\wf{2}$, and set $w\before w'$ if either $w_{12}<_{\lex} w'_{12}$
or $w_{12}=w'_{12}$ and $w<_{\lex} w'$. This ordering is different
from any of the $3!$ lexicographic orderings, and is stable under restriction
to subcubes.

To describe the class of linear orderings that generalize the lexicographic ordering
for $k\geq 3$, we need a couple of auxiliary definitions. A \emph{Schr\"oder tree}
is a rooted plane\footnote{A plane tree is a tree in which children of a node
are ordered.} tree each of whose internal nodes has at least~$2$ children. A \emph{weakly decreasing
Schr\"oder tree} is a Schr\"oder tree with a binary relation $\prece$ on the set of internal nodes that satisfies:
\begin{enumerate}[label=(ST\arabic*)]
\item $\prece$ is a total preorder, i.e., $\prece$ is transitive, reflexive, and $a\prece b$ or $b\prece a$ for every
two nodes $a,b$.
\item Every path from the root is strictly decreasing.
\end{enumerate}

\begin{center}
\begin{tikzpicture}[level 1/.style={sibling distance=7em},level 2/.style={sibling distance=5em}]
\scoped[local bounding box=tree,every node/.style={draw,circle}]
\node (123456) {$4$}
   child { node (123) {$3$}
     child { node (12) {$1$}
       child { [fill] circle (2pt) }            
       child { [fill] circle (2pt) }            
     }
     child { [fill] circle (2pt) }     
   }
   child  { node (4567) {$2$}
     child [sibling distance = 2em] { [fill] circle (2pt) }            
     child [sibling distance = 2em] { node(56) {$1$} 
       child {[fill] circle (2pt)}
       child {[fill] circle (2pt)}
     }            
     child [sibling distance = 2em] { [fill] circle (2pt) }            
   };
\scoped[local bounding box=labeled tree,every node/.style={draw,circle},leaf/.style={below,draw=none,execute at begin node=$\scriptstyle\wf,execute at end node=$}, xshift=20em]
\node (1234567) {$4$}
   child { node (123) {$3$}
     child { node (12) {$1$}
       child { [fill] circle (2pt) node[leaf] {2} }
       child { [fill] circle (2pt) node[leaf] {4} }            
     }
     child { [fill] circle (2pt) node[leaf] {7} }     
   }
   child  { node (4567) {$2$}
     child [sibling distance = 2em] { [fill] circle (2pt) node[leaf] {1} }
     child [sibling distance = 2em] { node(56) {$1$} 
       child {[fill] circle (2pt) node[leaf] {5}}
       child {[fill] circle (2pt) node[leaf] {6}}
     }            
     child [sibling distance = 2em] { [fill] circle (2pt) node[leaf] {3} }            
   };
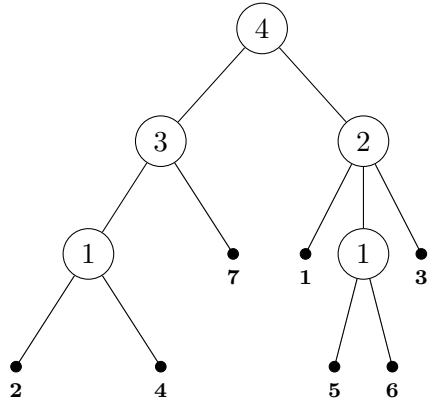
\node (labeled caption) [node distance=8ex, below=of labeled tree,text width=5cm,anchor=south] {Same, with leaves ordered by
$\wf{2}<\wf{4}<\wf{7}<\wf{1}<\wf{5}<\wf{6}<\wf{3}$};
\node [node distance=1ex,anchor=south] at (tree|-labeled caption) {A weakly decreasing Schr\"oder tree};
\end{tikzpicture}
\end{center}

For a weakly decreasing Schr\"oder tree $T$, we denote by $\prece_T$ the linear ordering on the internal
nodes of $T$. If $T$ has $k$ leaves, and we have an linear ordering on $[k]$,
we identify leaves of $T$ with the elements of $[k]$ by labeling leaves in the increasing
order. For a pair of leaves $\{\wf{a},\wf{b}\}\in\binom{[k]}{2}$
we write $[\wf{a},\wf{b}]_T$ for the bottommost node of $T$ that contains both $\wf{a}$ and~$\wf{b}$.

Given a linear ordering $<$ of $[k]$ and a weakly decreasing Schr\"oder tree~$T$ with $k$ leaves,
we can define a linear ordering~$\before_T$ on~$[k]^n$ as follows. Given two words $w,w'\in [k]^n$, we
consider all indices $i\in [n]$ such that $w_i\neq w_i'$. Among these, 
we pick the smallest $i$ such that $[w_j,w_j']_T \prece_T [w_i,w_i']_T$ for all $j\in [n]$.
We then declare $w\before_T w'$ if $w_i<w_i'$.

\begin{theorem}\label{thm:many}
For every $k$ and $d$ there exists $n$ with the following property: for every linear ordering $\before$
of $[k]^n$ there is a \D-subcube $C$ of $[k]^n$ such that the restriction of $\before$ to $C$ 
is equal to $\before_T$ for some weakly decreasing Schr\"oder tree $T$ with $k$ leaves and some linear ordering of~$[k]$.
\end{theorem}

\parshape=13 0cm.86\hsize 0cm.86\hsize 0cm.86\hsize 0cm.86\hsize 0cm.86\hsize 0cm\hsize 0cm\hsize 0cm\hsize 0cm\hsize 0cm\hsize 0cm\hsize 0cm\hsize 0cm \hsize
For
\vadjust{\hfill\smash{\raise -11ex\llap{%
\tikz[scale=0.5,every node/.style={draw,circle,transform shape}] 
{
 \node (123) {$2$}
 child { node {$1$}
   child {[fill] circle (2pt)}
   child {[fill] circle (2pt)}
 } 
 child {[fill] circle (2pt)}
 ;
}\quad}}}
example, the ordering on $[3]^n$ above is obtained from the tree depicted on the right,
under the usual ordering $\wf{1}<\wf{2}<\wf{3}$. Another ordering on $[3]^n$ can be obtained by taking a mirror image of the tree on the right.
The usual lexicographic ordering is obtained by taking $T=
\tikz[baseline=-1.1ex,scale=0.38,every node/.style={draw,circle,transform shape},level distance=3ex,sibling distance=1em] 
{
  \node (123) {}
  child {[fill] circle (2pt)}
  child {[fill] circle (2pt)} 
  child {[fill] circle (2pt)};
}$. In general, Bodini, Genitrini and Naima \cite[Section 3.2]{bodini_genitrini_naima} showed that the number 
of weakly decreasing Schr\"oder trees with $k$ leaves is 
equal to the $(k-1)$'st ordered Bell number, and hence is asymptotic to $\frac{(k-1)!}{2(\ln 2)^k}$.

\paragraph{An extension of Ramsey's theorem.}
Interestingly, after proving the main result in this paper (\autoref{thm:many}), we found that 
there is another way to prove \autoref{thm:es}, which relies on an extension of Ramsey's theorem.
\begin{theorem}\label{thm:ramsey}
For every $d$ and $r$, there exists $m$ such that for every $r$\nobreakdash-edge-coloring of the complete graph on $[m]$
there is a monochromatic proper affine \D-cube.
\end{theorem}
It is easy to deduce \autoref{thm:es} from \autoref{thm:ramsey}: Given a sequence $a_1,\dotsc,a_m$ 
we color edge $\{i,j\}$, with $i<j$, with one of two colors according to whether $a_i<a_j$ or~$a_j<a_i$.
A monochromatic clique in this coloring then corresponds to a monotone subsequence.

\paragraph{Paper organization.}
The bulk of the paper is occupied by the proof of \autoref{thm:many},
which is split into two parts. We first show that, for any linear ordering $\before$ of~$[k]^n$,
there is a large subcube $C$ such that the restriction of $\before$ to $C$ enjoys a certain symmetry property,
which we call \emph{uniformity}. For~$k=2$, the uniform linear orderings are then easily seen to be lexicographic, 
which proves \autoref{thm:two}. That is done in \autoref{sec:uniform}. The case of general $k$ requires a
more careful analysis of uniform linear orderings, which we carry out in \autoref{sec:general}. 
We conclude the paper with the proof of \autoref{thm:ramsey} in \autoref{sec:ramsey} and with some open problems
in \autoref{sec:open}.

\paragraph{Acknowledgment.} We thank James Cummings for comments on the earlier version of this paper.

\section{Uniform linear orderings}\label{sec:uniform}
Given a linear ordering $\before$ on $[k]^n$, a restriction of $\before$ to a \D-subcube induces,
under the canonical bijection, one of~$(k^d)!$ linear orderings on~$[k]^d$. We say that $\before$
is \emph{\D-uniform} if all restrictions to \D-subcubes induce the same linear ordering
on~$[k]^d$. An ordering that is $d$-uniform for all $d$ is called simply \emph{uniform}.

\begin{lemma}\label{lem:uniform}
For every $k$ and $d$ there exists $n$ with the following property: for every linear ordering
$\before$ of $[k]^n$ there is a \D-subcube $C$ of $[k]^n$ such that the restriction of $\before$ to $C$ 
is uniform.
\end{lemma}
We shall deduce this from the following special case of the Graham--Rothschild theorem \cite{graham_rothschild} 
(see also \cite{promel_voigt} for a transparent exposition and a short proof).
\begin{lemma}[Graham--Rothschild theorem, the trivial group case]
For every $d,D,k,t$ there exists $n=n(d,D,k,t)$ such that, for any $t$-coloring of \D-subcubes of $[k]^n$,
there is a $D$\nobreakdash-subcube of $[k]^n$ all of whose \D-subcubes are monochromatic.
\end{lemma}
\begin{proof}[Proof of \autoref{lem:uniform}]
Let $\chi$ be the coloring of \D-subcubes of $[k]^n$ in $(k^d)!$ many colors that
assigns to each \D-subcube $C$ the restriction of $\before$ to~$C$.
By the Graham--Rothschild theorem, if $n$ is large enough, there is a $(2d-1)$-subcube
$C_0$ of~$[k]^n$ on which~$\chi$ is monochromatic. We use the canonical bijection to identify
this subcube with $[k]^{2d-1}$. Let $C$ be the subcube of $C_0$ induced by the canonical word
$\wf{1}\dotsb \wf{1} \wf{*_1} \wf{*_2}\dotsb \wf{*_d}$ (with $d-1$ many $\wf{1}$'s).
We claim that the restriction of $\before$ to $C$ is uniform.

Indeed, let $C'$ be an arbitrary $d'$\nobreakdash-subcube $C'$ of $C$. We can then complete $C'$ to a \D-subcube $C_0'$ of $C_0$
in such a way that $C'$ is the subcube of $C_0'$ induced by the canonical word $\wf{1}\dotsb \wf{1}\wf{*_1}\dotsb\wf{*_{d'}}$.
In this way every $d'$\nobreakdash-subcube $C'$ is identified with the same $d'$\nobreakdash-subcube of $[k]^d$, and so
is ordered in the same way.
\end{proof}

If $\before$ is uniform, with slight abuse of notation, we think of $\before$ as a linear ordering
on each $[k]^d$ for $d=1,2,\dotsc$. In particular, $\before$ induces an ordering on $[k]$,
which we denote $<$. The following implies \autoref{thm:two}.
\begin{proposition}\label{prop:lex}
Let $\before$ be a uniform linear ordering on $[k]^n$, for $n\geq 3$. Suppose $\wf{a},\wf{b}\in [k]$ satisfy
$\wf{a}<\wf{b}$. Then the restriction of $\before$ to $\{\wf{a},\wf{b}\}^n$ is the lexicographic ordering
for the restriction of $<$ to $\{\wf{a},\wf{b}\}$.
\end{proposition}
\begin{proof}
We claim that $\wf{ab}\before \wf{ba}$. Indeed, if $\wf{ba}\before \wf{ab}$, then we reach a contradiction
by considering the sequence of inequalities
\[
  \wf{aab}\before \wf{bab} \before \wf{aba} \before \wf{aab},
\]
where the first inequality follows because $a\before b$, whereas the last two follow because $\wf{ba}\before \wf{ab}$.

To show that the restriction of $\before$ to $\{\wf{a},\wf{b}\}^n$ coincides with the lexicographic ordering
on $\{\wf{a},\wf{b}\}^n$, it suffices to show $w\before w'$ whenever $w'$ is the successor of $w$
in the lexicographic ordering. If $w=\wf{a}^n$, then $w'=\wf{a}^{n-1}\wf{b}$,
and $w\before w'$ because $\wf{a}\before \wf{b}$. Otherwise $w=w_0\wf{ab}^t$ and
$w'=w_0\wf{ba}^t$ for some nonnegative integer $t$ and a word~$w_0$. In that case
$w\before w'$ because $\wf{ab}^t\before \wf{ba}^t$ follows from $\wf{ab}\before \wf{ba}$.
\end{proof}

\section{The proof of \autoref{thm:many}}\label{sec:general}
By \autoref{lem:uniform}, it suffices to show that every uniform linear ordering is of the form
$\before_T$ for some weakly decreasing Schr\"oder tree $T$ and some linear ordering $<$ on~$[k]$. 
By permuting the elements of $[k]$ if necessary, we may assume that the restriction of the uniform linear ordering
to $[k]$ coincides with the usual ordering on~$\N$. Hence, in this section $<$ denotes the usual ordering on~$[k]$.

Call a subinterval of $[k]$ \emph{nontrivial} if it is of length at least $2$.
Given a linear ordering $\before$ on $[k]^n$, we define a binary relation $\prece$ on nontrivial subintervals of $[k]$
as follows. For any $\wf{a},\wf{b},\wf{c},\wf{d}\in [k]$ with $\wf{a}< \wf{b}$ and $\wf{c}< \wf{d}$,
we write $[ \wf{a},\wf{b}] \prece [ \wf{c},\wf{d}]$ if
$\wf{cb}\before \wf{da}$. If both $[ \wf{a},\wf{b}] \prece [ \wf{c},\wf{d}]$
and $[ \wf{c},\wf{d}] \prece [ \wf{a},\wf{b}]$ hold, then we write 
$[ \wf{a},\wf{b} ] \approx [ \wf{c},\wf{d} ]$.

Note that if $\mathord{\before}=\mathord{\before}_T$, then $[ \wf{a},\wf{b} ] \prece [ \wf{c},\wf{d} ]$ holds
if and only if $[ \wf{a},\wf{b}]_T \prece_T [ \wf{c},\wf{d}]_T$ holds. 
Furthermore, if $\mathord{\before}=\mathord{\before}_T$, then $\prece$ satisfies the following three properties:
\begin{enumerate}
\item Transitivity: $[ \wf{a},\wf{b} ] \prece [ \wf{c},\wf{d} ]$ and
$[ \wf{c},\wf{d} ] \prece [ \wf{e},\wf{f} ]$ together imply
$[ \wf{a},\wf{b} ] \prece [ \wf{e},\wf{f} ]$.
\item Comparability: $[ \wf{a},\wf{b} ] \prece [ \wf{c},\wf{d} ]$ or $[ \wf{c},\wf{d} ] \prece [ \wf{a},\wf{b} ]$.
\item Ultrametric property: if $\wf{a}<\wf{b}<\wf{c}$, then $[ \wf{a},\wf{c} ]\approx \max([ \wf{a},\wf{b} ],[ \wf{b},\wf{c} ])$.
\end{enumerate}
Call any relation $\prece$ on nontrivial subintervals of $[k]$ \emph{tree-like} if it satisfies these three properties.
Given a uniform linear ordering on $[k]^n$, we first show that $\prece$ is tree-like, and then use $\prece$ to build
a weakly decreasing Schr\"oder tree. That is done in the next two lemmas. Then in \autoref{lem:ident}, we show that the ordering
induced by the resulting tree (almost) coincides with the original ordering on~$[k]^n$.

\begin{lemma} Suppose $n\geq 3$. If a linear ordering $\before$ on $[k]^n$ is uniform,
then $\prece$ is tree-like.
\end{lemma}
\begin{proof}
\textit{Transitivity:} By the assumption we have $\wf{cb}\before \wf{da}$ and $\wf{ed}\before \wf{fc}$.
From this it follows that $\wf{edb}\before \wf{fcb}\before \wf{fda}$, which implies that $\wf{eb}\before \wf{fa}$ by uniformity.

\textit{Comparability:} Suppose $[ \wf{a},\wf{b} ] \not \prece [ \wf{c},\wf{d} ]$, and so $\wf{da}\before \wf{cb}$.
Also, \autoref{prop:lex} tells us that $\{\wf{c},\wf{d}\}^n$ is ordered lexicographically, implying that $\wf{cbd}\before\wf{dbc}$. 
Hence, $\wf{dad}\before \wf{cbd}\before \wf{dbc}$, which is to say $\wf{ad}\before \wf{bc}$.

\textit{Ultrametric property:} We first show that $\max([ \wf{a},\wf{b} ],[ \wf{b},\wf{c} ])\prece [ \wf{a},\wf{c} ]$.
We have $\wf{ab}\before\wf{ac}\before\wf{ca}$, and so $[ \wf{a},\wf{b} ]\prece[ \wf{a},\wf{c} ]$. Similarly,
$\wf{ac}\before \wf{bc}\before \wf{cb}$, and so $[ \wf{b},\wf{c} ]\prece [ \wf{a},\wf{c} ]$.

We next show that $[ \wf{a},\wf{c} ]\prece\max([ \wf{a},\wf{b} ],[ \wf{b},\wf{c} ])$.
Suppose $[ \wf{a},\wf{c} ]\not\prece [ \wf{a},\wf{b} ]$. Then $\wf{ba}\before \wf{ac}$, and so
$\wf{bca}\before\wf{acc}\before \wf{caa}$, and hence $[ \wf{a},\wf{c} ]\prece [ \wf{b},\wf{c} ]$.
\end{proof}

For a relation $\prece$ on the nontrivial subintervals of $[k]$ and a weakly decreasing Schr\"oder tree
$T$, we abuse notation and write $\mathord{\prece}=\mathord{\prece}_T$ provided 
$[\wf{a},\wf{b}]\prece [\wf{c},\wf{d}]$ holds if and only if $[\wf{a},\wf{b}]_T\prece_T [\wf{c},\wf{d}]_T$ holds.

\begin{lemma}\label{lem:reconstruction}
For every tree-like ordering $\prece$ on $\binom{[k]}{2}$ there exists a 
weakly decreasing Schr\"oder tree $T$ such that $\mathord{\prece}=\mathord{\prece}_T$.
\end{lemma}
In a weakly decreasing Schr\"oder tree $T$ we may have
$[\wf{a},\wf{b}]_T\approx_T [\wf{c},\wf{d}]_T$ for two reasons: either
because $[\wf{a},\wf{b}]_T=[\wf{c},\wf{d}]_T$, or neither of $[\wf{a},\wf{b}]_T$
and $[\wf{c},\wf{d}]_T$ is a descendant of one another, and they happen to be equal
in the $\prece_T$ preorder. Therefore, to build a tree out of a tree-like ordering,
we need to distinguish these two situations. We achieve this by identifying the node
with the widest interval that generates it.
\begin{proof}[Proof of \autoref{lem:reconstruction}]
For a nontrivial subinterval $[\wf{a},\wf{b}]$ 
of $[k]$, let $\wf{a'}$ be the least element of $[k]$ such that $[\wf{a'},\wf{b}]\approx [\wf{a},\wf{b}]$.
Likewise, let $\wf{b'}$ be the largest element of $[k]$ such that $[\wf{a},\wf{b'}]\approx [\wf{a},\wf{b}]$.
Define $\overline{[\wf{a},\wf{b}]}\eqdef [\wf{a'},\wf{b'}]$.

Let $\mathcal{N}\eqdef \{ \overline{[\wf{a},\wf{b}]} : \wf{a}\before\wf{b} \}$. We shall take $\mathcal{N}$ 
be the set of nodes of our tree. To show that we indeed obtain a tree, we must prove that
every two intervals from $\mathcal{N}$ are either disjoint or one of them contains the other. 
To do this we first show two basic properties of the map $[\wf{a},\wf{b}]\mapsto \overline{[\wf{a},\wf{b}]}$.\smallskip

\textbf{Claim 1:} \textit{$\overline{[\wf{a},\wf{b}]}\approx [\wf{a},\wf{b}]$ for every $\{\wf{a},\wf{b}\}\in\binom{[k]}{2}$.} Indeed, suppose $\overline{[\wf{a},\wf{b}]}=[\wf{a'},\wf{b'}]$.
If either $\wf{a}=\wf{a'}$ or $\wf{b}=\wf{b'}$, then the claim follows. Say $\wf{a'}<\wf{a}<\wf{b}<\wf{b'}$.
From the ultrametric property for the triple $\wf{a'}<\wf{a}<\wf{b'}$ we deduce that either
$[\wf{a'},\wf{b'}]\approx [\wf{a},\wf{b'}]$ or $[\wf{a'},\wf{b'}]\approx [\wf{a'},\wf{a}]$. In the former case
$[\wf{a},\wf{b}]\approx [\wf{a},\wf{b'}]\approx [\wf{a'},\wf{b'}]=\overline{[\wf{a},\wf{b}]}$, proving the claim.
In the latter case, two applications of the ultrametric property yield 
$[\wf{a}',\wf{b'}]\approx [\wf{a'},\wf{a}]\prece [\wf{a'},\wf{b}]\prece [\wf{a'},\wf{b'}]$, 
and so $[\wf{a'},\wf{b'}]\approx [\wf{a'},\wf{b}]\approx [\wf{a},\wf{b}]$ as well.\smallskip

\textbf{Claim 2:} \textit{$\overline{\overline{[\wf{a},\wf{b}]}}= \overline{[\wf{a},\wf{b}]}$ for every $\{\wf{a},\wf{b}\}\in\binom{[k]}{2}$.}
Indeed, suppose $\overline{[\wf{a},\wf{b}]}=[\wf{a'},\wf{b'}]$ and $\overline{\overline{[\wf{a},\wf{b}]}}=[\wf{a''},\wf{b''}]$.
Then by two applications of Claim~1 it follows that $[\wf{a''},\wf{b}]\prece [\wf{a''},\wf{b''}]\approx [\wf{a'},\wf{b'}]\approx [\wf{a},\wf{b}]$,
which, by the minimality of $\wf{a'}$, implies that $\wf{a''}=\wf{a'}$. Similarly, $\wf{b''}=\wf{b'}$.\smallskip

We are now ready to prove that every pair of intervals in $\mathcal{N}$ is either disjoint or comparable.
Let $[\wf{a},\wf{b}]$ and $[\wf{c},\wf{d}]$ be any two intervals from $\mathcal{N}$,
and suppose that they are not disjoint. Say $\wf{c}\leq \wf{b}$ (the case $\wf{a}\leq \wf{d}$ is analogous, and can be
reduced to this case by swapping the roles of the two intervals). If we also have $\wf{c}\leq \wf{a}$, then the interval
$[\wf{c},\wf{d}]$ contains $[\wf{a},\wf{b}]$. So, assume that $\wf{a}<\wf{c}$. By Claim~2 we may
assume that $\wf{a}$ is the minimal $\wf{a'}$ such that $[\wf{a'},\wf{b}]\approx [\wf{a},\wf{b}]$,
and similarly for $\wf{b},\wf{c},\wf{d}$. From the minimality of $\wf{c}$ we infer that $[\wf{a},\wf{d}]\not\approx [\wf{c},\wf{d}]$.
By the ultrametric property applied to the triple $\wf{a}<\wf{c}<\wf{d}$, it follows that $[\wf{a},\wf{d}]\approx [\wf{a},\wf{c}]$.
By another application of the ultrametric property, this time to $\wf{a}<\wf{c}<\wf{b}$, we infer that
$[\wf{a},\wf{c}]\prece [\wf{a},\wf{b}]$, and so
$[\wf{a},\wf{d}]\prece [\wf{a},\wf{b}]$. The ultrametric property of $\wf{a}<\wf{b}<\wf{d}$ then implies
that $[\wf{a},\wf{d}]\approx [\wf{a},\wf{b}]$, and so $\wf{b}\geq \wf{d}$ by the maximality of $\wf{b}$.
Hence, $[\wf{c},\wf{d}]$ is contained in $[\wf{a},\wf{b}]$.

It follows that intervals in $\mathcal{N}$ naturally form a tree under the containment relation.
The tree is plane, with intervals ordered in the natural way. We add leaves to the tree by
declaring that leaf $\wf{a}$ is a descendant of all intervals that contain~$\wf{a}$.
For tree nodes $[\wf{a},\wf{b}],[\wf{c},\wf{d}]\in\mathcal{N}$, we order them $[\wf{a},\wf{b}]\prece_T [\wf{c},\wf{d}]$
if and only if $[\wf{a},\wf{b}]\prece [\wf{c},\wf{d}]$.
The ultrametric property then ensures that every path from the root is decreasing.
Denote the resulting weakly decreasing Schr\"oder tree by~$T$.

Recall that $[\wf{a},\wf{b}]_T$ is the bottommost node of $T$ containing $\wf{a}$ and~$\wf{b}$.
Since $\wf{a},\wf{b}\in \overline{[\wf{a},\wf{b}]}$, it follows that
$[\wf{a},\wf{b}]_T\subseteq \overline{[\wf{a},\wf{b}]}$. On the other hand,
$[\wf{a},\wf{b}]\subseteq [\wf{a},\wf{b}]_T$, which implies that
$\overline{[\wf{a},\wf{b}]}\subseteq \overline{[\wf{a},\wf{b}]_T}=[\wf{a},\wf{b}]_T$.
So, $[\wf{a},\wf{b}]_T = \overline{[\wf{a},\wf{b}]}$ for every $\{\wf{a},\wf{b}\}\in\binom{[k]}{2}$
from which $\mathord{\prece}=\mathord{\prece}_T$ follows.
\end{proof}

The last ingredient in the proof of \autoref{thm:many} is the next result. 
\begin{lemma}\label{lem:ident}
If $T$ is a weakly decreasing Schr\"oder tree with $k$ leaves, and $\before$ is a uniform linear ordering on $[k]^n$ such that $\mathord{\prece} = \mathord{\prece}_T$,
then $\before$ is equal to $\before_T$ on every $(n-1)$\nobreakdash-subcube of~$[k]^n$.
\end{lemma}
Note that, for the reason that will become clear from the proof, we do not assert that $\mathord{\before} = \mathord{\before}_T$. 
\autoref{thm:many} nonetheless follows as we may restrict to a subcube of one dimension smaller.
\begin{proof}
It suffices to show that $w\before w'$ whenever $w\before_T w'$ and $w,w'$ differ in $t\leq n-1$ positions.
The proof is by induction on $t$. The case $t=1$ holds because our assumption that the ordering on $[k]$ is the same for 
$\before$ and~$\before_T$. The case $t=2$ holds because $\mathord{\prece} = \mathord{\prece}_T$. So, assume that $t\geq 3$.

For ease of notation we identify the $t$\nobreakdash-subcube of $[k]^n$ containing both $w$ and $w'$ with $[k]^t$.
This way, $w$ and $w'$ differ in every position. Let $i$ be the smallest natural number 
such that $[w_j,w_j']_T \prece_T [w_i,w_i']_T$ for all $j\in [n]$. Note that $w_i<w_i'$ because $w\before_T w'$.\smallskip

The symbol $w_i$ breaks $w$ into three parts, the prefix, the symbol $w_i$ itself, and the suffix. The prefix and
the suffix cannot be both empty. Suppose first that the prefix is non-empty; we then write $w$ and $w'$ as
\begin{align*}
w&=\prefw\phantom{'} w_{i-1}w_i\suffw,\\
w'&=\prefw' w_{i-1}'w_i'\suffw'
\end{align*}
for some words $\prefw,\prefw'\in [k]^{i-2}$ and $\suffw,\suffw'\in [k]^{n-i}$.
If $w_{i-1}\before w_{i-1}'$, then $w\before \prefw w_{i-1}' \suffw$ and
since $\prefw w_{i-1}' \suffw\before w'$ by the induction hypothesis,
the inequality $w\before w'$ follows. So, we may assume that $w_{i-1}'\before w_{i-1}$.

Because $t\leq n-1$ and $\before $ is uniform,
the inequality $w\before w'$ will follow once we show that 
$\overline{w}\before \overline{w}'$, where
\begin{align*}
\overline{w}&\eqdef\prefw\phantom{'} w_{i-1}w_iw_i\suffw,\\
\overline{w}'&\eqdef\prefw' w_{i-1}'w_i'w_i'\suffw'.
\end{align*}

The definition of $i$ implies that $[w_i,w_i']_T\not\prece_T [w'_{i-1},w_{i-1}]_T$. Hence, 
$[w_i,w_i']\not\prece [w'_{i-1},w_{i-1}]$, which is to say $w_{i-1}w_i\before w'_{i-1}w_i'$.
Therefore,
$
  \overline{w}\before \prefw w'_{i-1}w_i' w_i\suffw
$.
On the other hand, $\prefw w'_{i-1}w_i' w_i \suffw' \before \overline{w}'$ follows from the uniformity of $\before$ and the induction hypothesis applied to the words
$\prefw w_i\suffw$ and $\prefw' w_i' \suffw'$. Together these imply $\overline{w}\before \overline{w}'$. \smallskip

If the prefix of $w$ (before $w_i$) is empty, we write
\begin{align*}
w&=w_iw_{i+1}\suffw,\\
w'&=w_i'w_{i+1}'\suffw',
\end{align*}
and define
\begin{align*}
\overline{w}&=w_iw_iw_{i+1}\suffw,\\
\overline{w}'&=w_i'w_i'w_{i+1}'\suffw'.
\end{align*}
Because $[w'_{i+1},w_{i+1}]_T \prece_T [w_i,w_i']_T$, we have $[w'_{i+1},w_{i+1}] \prece [w_i,w_i']$,
and so $w_iw_{i+1}\before w_i'w_{i+1}'$. Since the induction hypothesis tells us that 
$w_i\suffw\before w_i'\suffw'$, we have $\overline{w}\before \overline{w}'$,
and so $w\before w'$ in this case as well.
\end{proof}

\section{Extension of Ramsey's theorem}\label{sec:ramsey}
In this section, we prove \autoref{thm:ramsey}. 

Let $n=n(2d,2,2,2)$ be as in the Graham--Rothschild theorem,
and set $m=3^n$. Let $\chi\colon \binom{[m]}{2}\to [r]$ be an
$r$\nobreakdash-coloring of the edges of~$K_m$. Define the projection map $\pi\colon [2]^n\to [m]$
by $\pi(w)\eqdef\sum_{i=1}^n w_i3^{n-i}$. The coloring $\chi$ of $\binom{[m]}{2}$ then induces a
coloring $\chi'$ of $\binom{[2]^n}{2}$ via $\chi'(w,w')\eqdef \chi\bigl(\pi(w),\pi(w')\bigr)$.
We can then define a $2$\nobreakdash-coloring of $2$\nobreakdash-subcubes of~$[2]^n$ as follows.
Let~$C$ be any $2$\nobreakdash-subcube. We identify it with~$[2]^2$ with the aid of the canonical bijection.
Then $\chi''(C)$ is equal to the $\chi'$\nobreakdash-color of the edge~$\{\wf{01},\wf{10}\}$.

By the Graham--Rothschild theorem, there is a $2d$\nobreakdash-subcube $C$ on which $\chi''$ is monochromatic.
Call pair of words $w,w'$ \emph{incomparable} if there exist both $i\in [n]$ such that $(w_i,w_i')=(\wf{0},\wf{1})$
and $j\in [n]$ such that $(w_j,w_j')=(\wf{1},\wf{0})$. Since $\chi''$ is monochromatic,
every two incomparable words in~$C$ are of the same color.

Identify $C$ with $[2]^{2d}$, and consider the set 
\[
  S\eqdef \{w\in [2]^{2d} : w_{2i-1}\neq w_{2i}\text{ for all }i=1,2,\dotsc,d\}.
\]
Though $S$ is not a \D-subcube, its image under the map $\pi$ is an affine \D-cube.
Since every two words in $S$ are incomparable, it follows that $\pi(S)$ 
is monochromatic.

\section{Open problems}\label{sec:open}
\begin{itemize}
\item Conlon and Kam\v{c}ev \cite{conlon_kamcev} showed that for every $r$\nobreakdash-coloring of $[3]^n$ there are monochromatic
lines whose wildcard set is a union of at most $r$~intervals (see also \cite{leader_raty,kamcev_spiegel} for a strengthening
for even~$r$). We do not know if one can find a combinatorial line whose wildcard set is an arithmetic progression.

\item In this paper we made no effort to obtain good quantitative bounds. The right dependence of $m$ on $d$
in \autoref{thm:es} is probably doubly exponential.
\end{itemize}

\bibliographystyle{plain}
\bibliography{hjordering}
\end{document}